%% file: arxiv_main.tex
\documentclass[uplatex,dvipdfmx,11pt]{article}
\usepackage[margin=25mm]{geometry}

\input{preemble}
\input{commands}

\begin{document}
\title{\textbf{Infinite collision property for the three-dimensional uniform spanning tree }}
	\author{Satomi Watanabe}
	\date{Department of Advanced Mathematical Sciences \\ Graduate School of Informatics \\ Kyoto University}
	\maketitle
	
\begin{abstract}
	Let $\ust$ be the uniform spanning tree on $\zzz$, whose probability law is denoted by $\prob$. 
	For $\prob$-a.s. realization of $\ust$, the recurrence of the the simple random walk on $\ust$ is proved in 
		\cite{BLPS01} and it is also demonstrated in \cite{Hut-Peres} that two independent 
		simple random walks on $\ust$ collide infinitely often. 
	In this article, we will give a quantitative estimate on 
	the number of collisions of two independent 
		simple random walks on $\ust$, which provides another proof of the infinite collision property
		of $\ust$. 
\end{abstract}

\section{Introduction}\label{sec-intro}
	\input{intro}

\section{Definitions and backgrounds}\label{sec-defi}
	\input{section2}

\section{Proof of the main theorem}\label{sec-thm}
	\input{section3}

\input{main.bbl}


\end{document}

%% file: preemble.tex
\usepackage[dvipdfmx]{graphicx}

\usepackage{amssymb}
\usepackage{amsmath}
\usepackage{amsfonts}
\usepackage{mathtools}
\usepackage{amsrefs}
\usepackage[normalem]{ulem}

\usepackage{enumerate}

\usepackage{amsthm}
	\theoremstyle{plain}
		\newtheorem{theo}{Theorem}[section]
		
		\newtheorem{prop}[theo]{Proposition}
		\newtheorem{defi}[theo]{Definition}
		\newtheorem{cor}[theo]{Corollary}
	\theoremstyle{definition}	
		\newtheorem{remark}[theo]{Remark}
\usepackage{latexsym}

\makeatletter
\@addtoreset{equation}{section}

\makeatother

\usepackage[hang,small,bf]{caption}
\usepackage[subrefformat=parens]{subcaption}
\DefineSimpleKey{bib}{archivePrefix}{}
\BibSpec{arXiv}{%
  +{}{\PrintAuthors}{author}
  +{,}{ \textit}{title}
  +{}{ \parenthesize}{date}
  +{,}{ arXiv: }{eprint}
}

%% file: commands.tex
\newcommand{\zzz}{\mathbb{Z}^3}
\newcommand{\zd}{\mathbb{Z}^d}

\newcommand{\prob}{\mathbf{P}}
\newcommand{\ust}{\mathcal{U}}
\newcommand{\Reff}[2]{R_{\mathrm{eff}}({#1},{#2})}

\newcommand{\gamfty}{\gamma_\infty}
\newcommand{\vep}{\varepsilon}

\newcommand{\ie}{\textit{i}.\textit{e}.\ }

%% file: intro.tex


The aim of this article is to investigate the collision property of two independent simple random walks 
	on the three-dimensional uniform spanning tree. 
Let us first begin with the introduction of uniform spanning forests on $\zd$. 
 If $G_n$ is a sequence of finite subgraphs which exhausts $\zd$, then it is proved by Pemantle \cite{P91} that
	the sequence of the uniform spanning measures on $G_n$  converges weakly to a 
	probability measure which is supported on the set of spanning forests of $\zd$. 
Pemantle \cite{P91} also showed that the uniform spanning forest is a single tree almost surely, which is called the uniform spanning tree (UST) 
	on $\zd$, for $d\le 4$, while it is not a tree but a spanning forest with infinitely many connected components when $d\ge 5$. 
Since their introduction, study of uniform spanning forests has played an important role in the 
	progress of probability theory, because of its connection to various probabilistic models, see  \cite{BLPS01} for details.

The behavior of random walks on the uniform spanning forests on $\zd$ strongly depends on the 
	dimension $d$ in terms of spectral and geometric properties of the forests. 
In particular, it is proved that the random walk exhibits mean-field behavior for $d\ge 4$, 
	with precise logarithmic corrections in $d=4$ \cites{Hal-Hut, Hut-2}. 
On the other hand, for $d\le 3$, different exponents appear in the asymptotic behavior of
	several quantities such as transition density and mean-square displacement of the 
	random walk \cites{ACHS20, BM}. 
This is already confirmed at least for $d=2$ and it is strongly believed that this is the case in three dimensions. 
Further detailed estimates on the random walk on the uniform spanning tree have been established 
	for $d=2$ \cite{BCK21} and $d=3$ \cite{SWarx}. 

In this article, we will estimate the number of collisions of two independent random walks on the three-dimensional 
	uniform spanning tree. 
To be more precise, let us introduce some terminology here. 
For infinite connected recurrent graph $G$, let $X$ and $Y$ be independent (discrete time) 
	simple random walks on $G$. 
We say that $G$ has the infinite collision property 
	when $|\{n\mathrel{:}X_n=Y_n\}|=\infty$ holds almost surely, where $|A|$ denotes the cardinality of $A$. 
For classical examples such as $\mathbb{Z}$ and $\mathbb{Z}^2$, it is easy to see that 
	two independent simple random walks collide infinitely often. 
On the other hand, Krishnapur and Peres \cite{Kris-Peres} gave an example of a recurrent graph 
	for which the number of collision is almost surely finite. 
For collisions on random graphs, Barlow, Peres and Sousi \cite{collision} proved that 
	a critical Galton-Watson tree, the incipient infinite cluster in high dimensions and 
	the uniform spanning tree on $\mathbb{Z}^2$ all have the infinite collision property almost surely. 
The infinite collision property of reversible random rooted graphs including uniform spanning trees on $\zd$ ($d\le 4$) 
	and every component of uniform spanning forests on $\zd$ ($d\ge 5$) was proved in \cite{Hut-Peres}. 
The purpose of this article is to give a quantitative estimate of the number of collisions until two random walks exit a ball of 
	the three-dimensional UST, which was not dealt with in \cite{Hut-Peres}.  
Collisions of random walks on various graphs can be useful to study complex networks, 
	and we hope it yields new discernment for industry and mathematics. 
The infinite collision property is a relatively new concept in contrast to recurrence/transience. 
It would be worth investigating its basic properties and application to wide areas and 
	giving various examples. 

Now we state the main result of this article. 
Let $\ust$ be the uniform spanning tree on $\zzz$ and $\prob$ be its law. 
Let $\widetilde{X}$ and $\widetilde{Y}$ be two independent simple random walks on $\ust$ killed when they exit 
	the intrinsic ball of $\ust$ of radius $r$. 
We denote by $P$ the law of $(\widetilde{X},\widetilde{Y})$ started at $(0,0)$ 
	and by $E$ the corresponding expectation. 
Let $Z_{B_r}$ be the total number of collisions of $\widetilde{X}$ and $\widetilde{Y}$ 
	(see Section \ref{section-2.4} for the precise definition). 
\begin{theo}\label{thm-first}
	For any $r\ge 1$ and $\vep>0$, there exist some universal constant $C>0$, $c>0$ and 
		some event $K(r,\vep)$ with $\prob(K(r,\vep))\ge 1-C\vep^c$ such that on $K(r,\vep)$,
	\begin{align}
		\vep r\le E(Z_{B_r}) \le 6r,	\label{1st-moment}	\\
		E(Z_{B_r}^2)\le 144r^2+6r,	\label{2nd-moment}
	\end{align}
	holds. In particular, on $K(r,\vep)$ we have 
	\begin{equation}\label{thm-e-1}
		P(\vep r \le Z_{B_r}\le 72\vep^{-2} r)\ge \vep^2/12.
	\end{equation}
\end{theo}
The infinite collision property of the three-dimensional UST directly follows from Theorem \ref{thm-first}. 
\begin{cor}\label{mainthm}
	The uniform spanning tree on $\zzz$ has the infinite collision property $\prob$-a.s. 
\end{cor}
\begin{remark}
	Note that the above statement includes two different probability measures, the law of the 
		three-dimensional UST and that of random walks on it. 
	Corollary \ref{mainthm} claims that if we choose a tree according to the law of the 
		three-dimensional UST and check whether two independent simple 
		random walks on the tree collide infinitely often almost surely, then it has the  infinite collision property 
		almost surely with respect to UST measure. 
\end{remark}
\begin{remark}
	In \cite{Hut-Peres}, it is proved that the uniform spanning tree on $\zd~(d=3,4)$ and 
		each connected component of the uniform spanning forest on $\zd~(d\ge 5)$ have 
		the infinite collision property. 
	In Section \ref{sec-thm} of this article, we will derive Corollary \ref{mainthm} from 
		Theorem \ref{thm-first}, which gives another proof for the three-dimensional case. 
	We expect that quantitative moment estimates of the number of collisions for the case $d\ge 4$
		can also be derived from various estimates obtained in \cite{Hal-Hut} and \cite{Hut-2}. 
	We will not pursue this further in the present article. 
\end{remark}

Let us briefly explain the strategy of the proof of the Theorem \ref{thm-first}. 
In order to bound the moments of $Z_{B_r}$, we will rewrite it in terms of the effective resistance of the three-dimensional UST, 
	which can be derived from some geometric properties of graphs. 
We will construct a ``good'' event and demonstrate that the three-dimensional UST exhibits such properties with 
	high probability. 

Before we end this section, let us explain the organization of this article. 
General notation together with backgrounds of the three-dimensional UST 
	and collisions of random walks will be introduced in Section \ref{sec-defi}. 
Then Theorem \ref{thm-first} and Corollary \ref{mainthm} will be proved in Section \ref{sec-thm}. 

\vspace{2ex}
\noindent \textbf{Acknowledgements:}
The author would like to thank Professor Daisuke~Shiraishi for helpful discussions on the proof of the main theorem and 
	 a careful reading of the article. 
The author would also like to thank Professor David~A.~Croydon for valuable comments on the analysis of uniform spanning trees. 
The author thank Professor Yuval~Peres for informing the author of their paper on the infinite collision 
	property \cite{Hut-Peres}.
The author is supported by JST, the Establishment of University Fellowships Towards the Creation of Science Technology Innovation, Grant Number JPMJFS2123.

%% file: section2.tex

In this section, we introduce the uniform spanning tree on connected graphs and an algorithm to construct the uniform 
	spanning tree on $\zzz$ with loop-erased random walk paths. 

First we introduce some notation for subsets of $\zzz$. 
For a set $A\subset\zzz$, let 
\begin{align*}
	\partial_i A&=\{x\in A\mathrel{:} \mbox{there exists }y\in\zzz\setminus A \mbox{ such that }d_E(x,y)=1\},	\\	
	\partial A&=\{x\in\zzz\setminus A\mathrel{:} \mbox{there exists }y\in A \mbox{ such that }d_E(x,y)=1\}.
\end{align*}
be the inner and outer boundary of $A$, respectively. 

For two points $x,y\in\zzz$, we let $d_E(x,y)$ be the Euclidean distance between $x$ and $y$. 
For $x\in\zzz$ and a connected subset $A\subset\zzz$, 
	we let $\mathrm{dist}(x,A)=\inf\{d_E(x,y) \mathrel{:} y\in A\}$. 

For $x,y\in\zzz$, we write $x\sim y$ if $d_E(x,y)=1$. 
A finite or infinite sequence of vertices $\theta=(\theta_0,\theta_1,\cdots)$ is called
	a \textbf{path} if $\theta_{i-1}\sim \theta_i$ for all $i=1,2,\cdots$. 
For a finite path $\theta=(\theta_0,\cdots,\theta_k)$, 
	we define the length of $\theta$ to be $\mathrm{len}(\theta)=k$. 

\subsection{Uniform spanning tree}
In this subsection, we introduce the three-dimensional uniform spanning tree, 
	the model of interest of this article. 

A subgraph of a connected graph $G$ is called a \textbf{spanning tree} on $G$ 
	if it is connected and without cycle, and its vertex set is the same as that of $G$. 
If we denote by $\mathcal{T}(G)$ the set of all spanning trees on $G$, then for a finite (connected) graph $G$, 
	$\mathcal{T}(G)$ is also a finite set. 
In this case, a random tree according to the uniform measure on $\mathcal{T}(G)$ is called 
	the \textbf{uniform spanning tree (UST)} on a finite graph $G$. 
For $\zzz$, we can define the uniform spanning tree as the weak limit of the USTs on the 
	finite boxes $\zzz\cap [-n,n]^3$, see \cite{P91}. 
The uniform spanning tree on $\zzz$ is also called the three-dimensional uniform spanning tree. 

Let $(\Omega,\mathcal{F},\prob)$ be the probability space where the three-dimensional UST $\ust$ is defined 
	and let $\mathbf{E}$ be the corresponding expectation. 
Note that $\ust$ is a one-ended tree $\prob$-a.s. (see \cite{P91}). 
For any $x,y\in\zzz$, we write $\gamma(x,y)$ for the unique self-avoiding path from $x$ to $y$ in $\ust$. 
	For $x\in\zzz$ and a connected subset $A\subset\zzz$, we denote by $\gamma(x,A)$ 
	the shortest path among $\{\gamma(x,y)\mathrel{:}y\in A\}$ if $x\not\in A$ and 
	$\gamma(x,A)=\{x\}$ if $x\in A$. 
We let $\gamma(x,\infty)$ be the unique infinite self-avoiding path from $x$ in $\ust$. 
We denote the intrinsic metric on the graph $\ust$ by $d_\ust$, 
	\ie $d_\ust(x,y)=\mathrm{len}(\gamma(x,y))$. 

We define intrinsic balls of $\ust$ by
\begin{equation}\label{intrball}
	B_\ust(x,r)=\{y\in\zzz \mathrel{:} d_\ust(x,y)\le r\}.
\end{equation}
Recall that $d_E$ stands for the Euclidean metric on $\zzz$. We denote Euclidean balls by 
\begin{equation}
	B(x,r)=\{y\in\zzz \mathrel{:} d_E(x,y)\le r \}.
\end{equation}

Now let us define the simple random walk on $\ust$. 
If we let $\mu_\ust(\{x\})$ be the number of edges of $\ust$ which contain $x$, then $\mu_\ust$ 
	gives a measure on $\zzz$.  
For a given realization of $\ust$, the simple random walk on $\ust$ is defined as the discrete time 
	Markov process $X^\ust=((X_n^\ust)_{n\ge 0},(P_x^\ust)_{x\in\zzz})$ 
	which jumps from its current location to a uniformly chosen neighbor in $\ust$ at each step. 
For $x\in \zzz$, we call $P_x^\ust$ the \textbf{quenched law} of the simple random walk on $\ust$. 

\subsection{Loop-erased random walk}
Now we define the loop-erased random walk, which plays an important role in the analysis on uniform spanning trees through an algorithm called Wilson's algorithm.  

For a finite path $\theta$, let us define the chronological loop erasure $\mathrm{LE}(\theta)$ of 
	$\theta$ as follows.
Let 
\[
	T(0)=\sup\{j\mathrel{:} \theta_j=\theta_0\}
\]
and $\widetilde{\theta}_0=\theta_{T(0)}$. 
Next we set 
\begin{equation}\label{defLE}
	T(i)=\sup\{j\mathrel{:}\theta_j=\theta_{T(i-1)+1}\},\ \ \ 
	\widetilde{\theta}_i=\theta_{T(i)},
\end{equation}
inductively. Finally let 
\[
	l=\inf\{j\mathrel{:}T(j)=k\}.
\]
Then $\mathrm{LE}(\theta)$ is defined by 
\[
	\mathrm{LE}(\theta)=(\widetilde{\theta}_0,\cdots,\widetilde{\theta}_l). 
\]
The loop erasure of a finite path $\theta$ is a simple path contained in $\theta$ and its 
	starting point and end point are the same as those of $\theta$. 
	
The \textbf{loop-erased random walk (LERW)} is the random simple path obtained as the 
	loop erasure of simple random walk path. 
The exact same definition as the loop erasure for finite paths is also applied to the infinite simple random walk (SRW) $S$ on $\zzz$.
Since $S$ is transient, the times $T(i)$ in (\ref{defLE}) are finite for every $i\in\mathbb{Z}$, 
	almost surely. 
The \textbf{infinite loop-erased random walk (ILERW)} is defined as the infinite random simple path $\mathrm{LE}(S)$. 

Next we introduce the growth exponent of the three-dimensional LERW, 
	which represents the time-space scaling of the LERW. 
We run the SRW on $\zzz$ started at the origin until the first exiting time of the ball of radius $n$ 
	centered at its starting point. 
Let $M_n$ be the length of the loop erasure of this SRW path. 
We denote the law of $S$ and the corresponding expectation by $P$ and $E$, respectively. 
If the limit
\begin{equation}\label{growthbeta}
	\beta\coloneqq\lim_{n\to\infty}\frac{\log E(M_n)}{\log n},
\end{equation}
exists, then this constant $\beta$ is called the growth exponent of the LERW. 
The existence of the limit is proved in \cite{S18} and that $\beta\in(1,5/3]$ is obtained in \cite{L99}. 
Although the exact value of $\beta$ has not been discovered yet, it is estimated that $\beta=1.624\cdots$ 
	by numerical calculations, see \cite{W10}. 
Moreover, following exponential tail bounds of $M_n$ are obtained in \cite{S18}. 
\begin{theo}(\cite[Theorem 1.1.4]{S18})
There exists $c>0$ such that for all $n\ge 1$ and $\kappa\ge 1$,
\[
	\prob(M_n\ge \kappa E(M_n))\le 2\exp\{-c\kappa\},
\]
and for any $\vep\in(0,1)$, there exist $0<c_\vep,C_\vep<\infty$ such that for all $n\ge 1$ and $\kappa\ge 1$, 
\begin{equation}\label{tail}
	P(M_n\le \kappa^{-1}E(M_n))\le C_\vep\exp\{-c_\vep\kappa^{\frac{1}{\beta}-\vep} \}.
\end{equation}
\end{theo}

\subsection{Wilson's algorithm}\label{subsec-WA}
Now let us recall Wilson's algorithm. 
Throughout the article, we denote by $S^z=\{S^z(n)\}_{n=0}^\infty$ the simple random walk on $\zzz$ 
	started at $z\in\zzz$ and denote its law by $P^z$.
We take $\{S^z\}_{z\in\zzz}$ to be independent. 

Wilson's algorithm is a method to construct UST with LERW paths. 
It was first established for finite graphs (see \cite{W96}) and then extended to 
	transient $\zd$ including $\zzz$ (see \cite{BLPS01}). 
Here we introduce Wilson's algorithm of the transient case. 
Let $\{v_1,v_2,\cdots\}\subset\zzz$ be an ordering of the vertices of $\zzz$ and 
	denote by $\gamfty$ the infinite LERW independent of $\{S^z\}$, started at the origin. 
For a path $\theta$ and a set $A\subset\zzz$, let 
\begin{equation}\label{timetau}
	\tau(A)=\tau_\theta(A)=\min\{i\ge 0\mathrel{:}\theta_i\in A\},
\end{equation}
be the first hitting time of the set $A$ for the path $\gamma$. 
We define a sequence $\{\ust_i\}$ of random subgraphs of $\zzz$ inductively by 
\begin{align*}
	\ust_0&=\gamfty,	\\
	\ust_i&=\ust_{i-1}\cup\mathrm{LE}(S^{z_i}[0,\tau_{S^{z_i}}(\ust_{i-1})]),\ i\ge 1,
\end{align*}
and finally let $\ust_\infty=\cup_i \ust_i$.
Then by \cite{BLPS01}, the law of the resulting random tree $\ust_\infty$ is the same as 
	that of the three-dimensional UST. 
It follows from this fact that the law of $\ust_\infty$ does not depend on the ordering 
	$\{v_1,v_2,\cdots\}$ of $\zzz$.

\subsection{Effective Resistance and Green's function}\label{section-2.4}
Finally we define the infinite collision property and introduce its connection to effective resistance 
	and Green's function. 

Let $G=(V,E)$ be a connected graph and let $X=\{X_n\}_{n=0}^\infty$ and 
	$Y=\{Y_n\}_{n=0}^\infty$ be independent discrete time simple random walks on $G$. 
For $x,y\in V$, we write $x\sim y$ if $x$ and $y$ are connected with an edge, \ie $\{x,y\}\in E$.  
We denote by $P_{a,b}$ the law of $\{(X_n,Y_n)\}_{n=0}^\infty$ with starting point 
	$(X_0,Y_0)=(a,b)$. 
\begin{defi}
	We define the total number of collisions between $X$ and $Y$ by 
	\[
		Z=\sum_{n=0}^\infty \mathbf{1}(X_n=Y_n). 
	\]
	Let $B$ be a connected subgraph of G and let $X^B=\{X_n^B\}_{n=0}^\infty$ and 
		$Y^B=\{Y_n^B\}_{n=0}^\infty$ be independent discrete time simple random walks on $G$ 
		killed when they exit $B$. 
	We define the total number of collisions of $X^B$ and $Y^B$ by 
	\begin{equation}\label{defZB}
		Z_B=\sum_{n=0}^\infty \mathbf{1}(X_n^B=Y_n^B). 
	\end{equation}
\end{defi}
\begin{defi}
	If 
	\begin{equation}\label{fcolli}
		P_{a,a}(Z<\infty)=1,
	\end{equation}
	holds for all $a\in G$, then $G$ has the \textbf{finite collision property}. 
	If 
	\begin{equation}\label{infcolli}
		P_{a,a}(Z=\infty)=1,
	\end{equation}
	holds for all $a\in G$, then $G$ has the \textbf{infinite collision property}. 
\end{defi}

\begin{remark}
	There is no simple monotonicity property for collisions. 
	Let $\mathrm{Comb}(\mathbb{Z})$ be the graph with vertex set $\mathbb{Z}\times\mathbb{Z}$
		and edge set 
	\begin{equation*}
		\{[(x,n),(x,m)]\mathrel{:}|m-n|=1\}\cup\{[(x,0),(y,0)]\mathrel{:}|x-y|=1\}.
	\end{equation*}
	Then $\mathrm{Comb}(\mathbb{Z})$ has the finite collision property (see \cite{collision}) and is a subgraph 
		of $\mathbb{Z}^2$, which has the infinite collision property. 
\end{remark}

It is proved that for any connected graph, either $(\ref{fcolli})$ or (\ref{infcolli}) holds. 
\begin{prop}(\cite[Proposition 2.1]{collision})
	Let $G$ be a (connected) recurrent graph. 
	Then for any starting point $(a,b)\in G\times G$ of the process $\{(X_n,Y_n)\}$, 
	\[
		P_{a,b}(Z=\infty)\in \{0,1\},
	\]
	holds. In particular, for all $a\in G$, either $P_{a,a}(Z=\infty)=0$ 
	or $P_{a,a}(Z=\infty)=1$ holds. 
\end{prop}

Next we define effective resistance and Green's function, which we make use of to estimate $Z_B$. 
\begin{defi}\label{def-Reff}
	Let $G=(V,E)$ be a connected graph and let $f$ and $g$ be functions on $V$. 
	Then we define a quadratic form $\mathcal{E}$ by 
	\[
		\mathcal{E}(f,g)=\frac{1}{2}\sum_{\substack{x,y\in V \\ x\sim y}} (f(x)-f(y))(g(x)-x(y)).
	\]
	If we consider $G$ as an electrical network by regarding each edge of $G$ to be a unit resistance, 
		then the \textbf{effective resistance} between disjoint subsets $A$ and $B$ of $V$ is defined by 
	\begin{equation}\label{defofreff}
		\Reff{A}{B}^{-1}=\inf\{\mathcal{E}(f,f)\mathrel{:}\mathcal{E}(f,f)<\infty, f|_A=1,f|_B=0\}.
	\end{equation}
	\end{defi}
If we let $\Reff{x}{y}=\Reff{\{x\}}{\{y\}}$, then $\Reff{\cdot}{\cdot}$ is a metric on $G$, see \cite{Wei19}. 

\begin{defi}\label{def-G}
	Let $B$ be a connected subgraph of $G$. 
	For the simple random walk $X^B$ on $G$ killed when it exits $B$, the \textbf{Green's function} 
		is defined by 
	\begin{equation}\label{defGB}
		G_B(x,y)=\sum_{n=0}^\infty P^x(X_n^B=y).
	\end{equation}
\end{defi}

Let $\mu_G(x)$ be the number of neighboring vertices of $x$ in $G$. 
Effective resistance and Green's function are related by the following equality: 
\begin{equation}\label{Reff-G}
	\mu_G(x)\Reff{x}{B^c}=G_B(x,x),
\end{equation}
see \cite{LP16} Section 2.2, for example. 

Now let us derive some estimates for $Z_B$ (see (\ref{defZB}) for definition) from 
	effective resistance and Green's function. 
For the first moment of $Z_B$, we have 
\begin{align}
	E_{0,0}(Z_B)&=\sum_{n=0}^\infty \sum_{x\in B}P_{0,0}(X_n^B=Y_n^B=x)	\notag \\
		&=\sum_{n=0}^\infty \sum_{x\in B}P^0(X_n^B=x)^2=\sum_{n=0}^\infty P^0(X_{2n}^B=0). \label{reff-g-1}
\end{align}
Since $P^x(X_{2n+1}^B=x)\le P^x(X_{2n}^B=x)$ for all $n$, we have that 
\[
	\frac{1}{2}G_B(x,x)=\frac{1}{2}\sum_{n=0}^\infty \left(P^x(X_{2n}^B=x)+P^x(X_{2n+1}^B=x)\right)\le \sum_{n=0}^\infty P^x(X_{2n}^B=x)\le G_B(x,x).
\]
Thus, it follows from (\ref{reff-g-1}) that 
\begin{equation}\label{reff-g-2}
	\frac{1}{2}G_B(0,0)\le E_{0,0}(Z_B)\le G_B(0,0).
\end{equation}

An upper bound of the second moment is obtained by 
\begin{align}
	E_{0,0}(Z_B^2)&=\sum_{n=0}^\infty \sum_{x\in B}P_{0,0}(X_n^B=Y_n^B=x)	\notag \\
	&+2\sum_{n=0}^\infty \sum_{m=n+1}^\infty \sum_{x\in B}\sum_{y\in B}P_{0,0}(X_n^B=Y_n^B=x,X_m^B=Y_m^B=y)	\notag \\
	&=E_{0,0}(Z_B)+2\sum_{n=0}^\infty P_{0,0}(X_n^B=Y_n^B=x)\sum_{m=0}^\infty P_{x,x}(X_n^B=Y_n^B=y)	\notag \\
	&\le G_B(0,0)+2G_B(0,0)\max_{x\in B}G_B(x,x), \notag
\end{align}
where we applied the Markov property for the second equality and (\ref{reff-g-2}) for the last inequality. 
By plugging (\ref{Reff-G}) into the above inequality, we obtain 
\begin{equation}\label{reff-g-3}
	E_{0,0}(Z_B^2)\le \mu_G(0)\Reff{0}{B^c}+2\mu_G(0)\Reff{0}{B^c}\max_{x\in B}\mu_G(x)\Reff{x}{B^c}.
\end{equation}



%% file: section3.tex


In this section, we will prove Theorem \ref{thm-first}. 
In order to do so, we will first estimate the effective resistance of $\ust$ between the origin and $\partial B(0,r)$ 
	in the following theorem. 

Let $U_r$ be the connected component of $\ust\cap B(0,r)$ which contains the origin. 
Recall that $\beta$ is the growth exponent of the three-dimensional LERW defined in (\ref{growthbeta}). 

\begin{theo}\label{effEball}
	There exists some universal constant $C>0$ such that for all $r\ge 1$ and $\lambda>0$, 
	\begin{equation}\label{Ebound}
		\prob(\Reff{0}{\mathcal{U}\setminus U_r}\ge r^\beta/\lambda^{1+4\beta})\ge 1-C\lambda^{-1}.
	\end{equation}
\end{theo}
\begin{proof}
	Note that it sufficies to prove the inequality (\ref{Ebound}) for $\lambda \ge \lambda_0$ where $\lambda_0$ 
		is a sufficiently large universal constant which does not depend on $r$. 

	We first fix $r>0$ and consider a sequence of subsets of $\zzz$ including $\partial_i B(0,r)$. 
	For $k=1,2,\cdots$, let $\delta_k=\lambda^{-1}2^{-k}$ and $\eta_k=(2k)^{-1}$. 
	We define $k_0$ to be the smallest positive integer such that $r\delta_{k_0}<1$. 
	Let 
	\[
		A_k=B(0,r)\setminus B(0,(1-\eta_k)r),
	\]
	and let $D_k$ be a finite subset of lattice points of $A_k$ with $|D_k|\le C\delta_k^{-3}$ such that 
	\[
		A_k\subset \bigcup_{z\in D_k} B(z,\delta_k r).
	\]
	
	Next we perform Wilson's algorithm rooted at infinity (see Section \ref{subsec-WA}) 
		to obtain the desired event of the three-dimensional UST. 
	Let $\ust_0=\gamfty$ \ie the infinite LERW started at the origin. 
	Given $\ust_k~(k\ge 0)$, we regard $\ust_k$ as the root of Wilson's algorithm and add branches started at 
		vertices in $D_{k+1}\setminus \ust_k$ and denote by $\ust_{k+1}$ the resulting random subtree at this step. 
	Once we obtain $\ust_{k_0}$, we add branches started at vertices in $\zzz\setminus \ust_{k_0}$ to 
		complete Wilson's algorithm. 
	Note that $\ust_k~(k=0,1,2,\cdots,k_0)$ is a subtree of $\ust$ containig all vertices in 
		$\bigcup_{i=1}^k D_i\cup\{0\}$ and the sequence $\{\ust_k\}_{k=0}^{k_0}$ is increasing. 
	Since $r\delta_{k_0}<1$, it holds that $\partial_i B(0,r)\subset D_{k_0}\subset \ust_{k_0}$. 
	
	Now we are ready to define the events where the effective resistance in (\ref{Ebound}) is bounded below. 
	First, we examine the behavior of the branches started at vertices contained in $D_1$. 
	For $z\in D_k~(k\ge 1)$, we denote by $y_z$ be the first point of $\ust_{k-1}$ visited by $\gamma(z,0)$ \ie 
		$d(z,y_z)=\min_{y\in \ust_{k-1}} d(z,y)$. 
	We define the event $F_z$ by 
	\begin{equation}\label{proof-0}
		F_z=\{\gamma(z,y_z)\cap B(0,\lambda^{-4}r)=\emptyset \},
	\end{equation}
	for $z\in D_1$. 
	Since $d_E(0,z)\ge r/2$, by \cite[Theorem 1.5.10]{L99}, 
		there exists some constant $C>0$ such that for all $\lambda\ge 2$, 
	\[
		\label{proof-1}
		\prob(F_z^c)\le \prob(S^z[0,\infty)\cap B(0,\lambda^{-4}r)\neq\emptyset)\le C\lambda^{-4},
	\]
	holds. 
	By taking the union bound, we obtain that 
	\begin{equation}\label{proof-2}
		\prob \left( \bigcup_{z\in D_1} F_z^c \right)\le |D_1|C\lambda^{-4}\le C\lambda^{-1},
	\end{equation}
	where the last inequality follows from the fact that $|D_1|\le C\lambda^3$. 
	
	Second, we bound from below the first time when $\gamfty$ exits $B(0,\lambda^{-4}r)$, 
		which is denoted by $\tau(B(0,\lambda^{-4}r)^c)$. 
	We define the event $\widetilde{F}$ by 
	\begin{equation}\label{proof-2.01}
		\widetilde{F}=\left\{\mathrm{len}\bigl(\gamfty[0,\tau(B(0,\lambda^{-4}r)^c)]\bigr)\ge r^\beta/\lambda^{1+4\beta}\right\}.
	\end{equation}
	By \cite[Theorem 1.1.4]{S18} and \cite[Corollary 1.3]{LS19}, there exist some constants $C>0$ and $c>0$ such that 
	\begin{equation}\label{proof-2.1}
		\prob(\widetilde{F}^c)\le C\exp\{-c\lambda^{1/2}\},
	\end{equation}
	for all $r\ge 1$ and $\lambda>0$. 
	
	Third, we consider the branches started at vertices in $D_k~(k\ge 2)$ step by step. 
	Let us begin by defining an event which guarantees ``hittability'' of $\gamma(x,\infty)$ for $x\in D_k$. 
	To be precise, for $x\in D_k~(k\ge 1)$ and $\xi>0$, we define the event $H_x(\xi)$ by 
	\begin{align*}
		H_x(\xi)=\Bigl\{ \mbox{There exists some } &z\in B(x,\delta_k r) \mbox{ such that }	\\
			&P^z(S^z[0,\tau_{S^z}(B(z,\delta_k^{1/2}r)^c)]\cap \gamma(x,\infty)=\emptyset)\ge \delta_k^\xi \Bigr\},
	\end{align*}
	where $S^z$ is an independent simple random walk started at $z\in\zzz$ and $P^z$ denotes its law. 
	Let
	\begin{equation}\label{proof-2.2}
		\widetilde{H}_k\coloneqq \bigcap_{x\in D_k}H_x(\xi_1)^c. 
	\end{equation}
	Note that 
	$P^z(S^z[0,\tau_{S^z}(B(z,\delta_k^{1/2}r)^c)]\cap \gamma(x,\infty)=\emptyset)$ is a function of $\gamma(x,\infty)$ 
		and thus $H_x$ and $\widetilde{H}_k$ are measurable with respect to $\ust_k$. 
	By \cite[Theorem 3.1]{SaSh18}, there exist some $C>0$ and $\xi_1$ such that 
	\begin{equation}\label{proof-3}
		\prob(H_x(\xi_1))\le C\delta_k^4\ \ \ \mbox{for all $r\ge 1$, $k\ge 1$ and $x\in D_k$},
	\end{equation}
	from which it follows that 
	\begin{equation}\label{proof-4}
		\prob(\widetilde{H}_k)\ge 1-|D_k|C\delta_k^4\ge 1-C'\delta_k,
	\end{equation}
	where $C'>0$ is uniform in $r\ge 1$ and $k\ge 1$. 

	Now we will demonstrate that conditioned on the event $\widetilde{H}_k$, 
		branches $\gamma(z,y_z)~(z\in D_{k+1})$ are included in $A_k$ with high conditional probability. 
	Let $M=\lceil 4/\xi_1\rceil$. 
	For $z\in D_{k+1}$, let 
	\[
		I_z=\left\{S^z[0,\tau_{S^z}(B(z,M\delta_k^{1/2}r))]\cap \ust_k=\emptyset \right\}.
	\]
	Since $z\in D_{k+1}\subset A_k$, we can take some $x\in D_k$ with $z\in B(x,\delta_k r)$ and 
		on the event $I_z$, we have that 
	\[
		S^z[0,T^1]\cap \gamma(x,\infty)=\emptyset,
	\]
	holds, where $T^1=\tau_{S^z}(B(z,\delta_k^{1/2}r))$. 
	
	In the rest of this proof, we take $\lambda\ge 6M$ without loss of generality. 
	Since $d_E(z,S^z(T^1-1))\le \delta_k r$, we have that $z_1\coloneqq S^z(T^1-1)\in A_k$ and 
		we can take $x_1\in D_k$ with $z_1\in B(x_1,\delta_k r)$. 
	By the same argument as the above, on the event $I_z$ we have that $S^z[T^1,T^2]\cap\gamma(x,\infty)=\emptyset$, 
		where $T^2=\tau_{S^z}(B(z_1,\delta_k^{1/2}r))$. 
	Iteratively, we obtain the sequences $\{T^i\}$, $\{z_i\}\subset A_k$ and $\{x_i\}\subset D_k~(i=1,2,\cdots, M)$ 
		and we have that 
	\[
		I_z\subset \bigcap_{i=1}^M \{R[T^{i-1},T^i]\cap \gamma(x_{i-1},\infty)=\emptyset \},
	\]
	where we set $T^0=0$ and $x_0=x$. 
	By the strong Markov property, it holds that 
	\begin{align*}
		P^z(I_z)&\le P^z\left(\bigcap_{i=1}^M \{R[T^{i-1},T^i]\cap \gamma(x_{i-1},\infty)=\emptyset \}\right)	\\
		&=\prod_{i=1}^M P^{z_{i-1}}(S^{z_{i-1}}[0,\tau_{S^{z_{i-1}}}(B(z_{i-1},\delta_k^{1/2}r))]\cap \gamma(x_{i-1},\infty)=\emptyset),
	\end{align*}
	from which it follows that 
	\[
		\widetilde{H}_k\subset \{P^z(I_z)\le \delta_k^4\}.
	\]
	Thus, by Wilson's algorithm, we have that for all $z\in D_{k+1}$, 
	\begin{equation}\label{proof-5}
		\prob\left(\gamma(z,y_z)\not\subset B(z,M\delta_k^{1/2}r)\mid \widetilde{H}_k\right)\le \delta_k^4.
	\end{equation}
	We define the event $\widetilde{I}_{k+1}$, which is measurable with respect to $\ust_{k+1}$, by 
	\begin{equation}\label{proof-5.1}
		\widetilde{I}_{k+1}=\bigcap_{z\in D_k+1} \left\{\gamma(z,y_z)\subset B(z,M\delta_k^{1/2}r)\right\}.
	\end{equation}
	Then by $(\ref{proof-5})$ and that $|D_{k+1}|\le C\delta_k^{-3}$, it holds that 
	\[
		\prob(\widetilde{I}_{k+1}\mid \widetilde{H}_k)\ge 1-|D_{k+1}|\delta_k^4\ge 1-C\delta_k.
	\]
	Combining this with (\ref{proof-4}), we obtain that 
	\begin{equation}\label{proof-6}
		\prob(\widetilde{H}_k\cap \widetilde{I}_{k+1})\ge 1-C\delta_k,
	\end{equation}
	for some universal constant $C>0$. 
	
	Finally we construct an event where the desired effective resistance bound holds. 
	\begin{figure}[tb]
		\centering
		\includegraphics[width=0.8\linewidth]{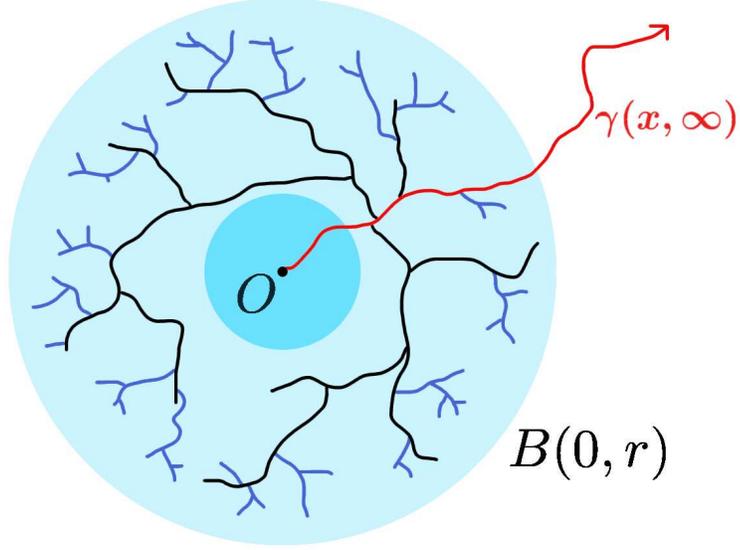}
		\label{eventK}
		\caption{In this figure, two circles represent Euclidean balls centered at the origin: the larger one is of radius 
		$r$ and the small one is of radius $\lambda^{-4}r$. 
		On the event $K$, the branches from $D_1$ do not enter the smaller ball of radius $\lambda^{-4}r$ and 
		branches from $D_k~(k\ge 2)$ hits the already constructed subtree $\ust_{k-1}$ before entering $B(0,r/2)$. 
		Moreover, the length of $\gamfty$ up to the exiting time $\tau(B(0,\lambda^{-4}r))$ is bounded below by 
		$r^\beta/\lambda^{1+4\beta}$.}
	\end{figure}
	Let 
	\begin{equation*}\label{proof-6.1}
		K=\left(\bigcap_{z\in D_1} F_z\right) \cap \widetilde{F} \cap \left(\bigcap_{k=1}^{k_0} (\widetilde{H}_k \cap \widetilde{I}_{k+1})\right).
	\end{equation*}
	Recall that $F_z$, $\widetilde{F}$, $\widetilde{H}_k$ and $\widetilde{I}_{k+1}$ are defined by 
		(\ref{proof-0}), (\ref{proof-2.01}), (\ref{proof-2.2}) and (\ref{proof-5.1}), respectively. 

	Then combining (\ref{proof-2}), (\ref{proof-2.1}) and (\ref{proof-6}), we obtain that 
	\begin{equation}\label{proof-7}
		\prob(K^c)\le C\lambda^{-1}+C\exp\{-c\lambda^{1/2}\}+\sum_{k=1}^\infty C\delta_k\le C\lambda^{-1}.
	\end{equation}
	We claim that on the event $K$, the following two statements hold: 
	\begin{enumerate}[(1)]\setlength{\leftskip}{2em}
		\item $d(0,y_z)\ge r^\beta/\lambda^{1+4\beta}$ for all $z\in D_1$.
		\item For $k\ge 2$, $\gamma(z,0)$ hits $\ust_1$ before entering $B(0,r/2)$ for all $z\in D_k$.
	\end{enumerate}
	Note that (1) is immideate from $K\subset (\bigcup_{z\in D_1} F_z)^c\cap \widetilde{F}^c$ and 
		(2) follows from $K\subset (\bigcap_{k=1}^{k_0} (\widetilde{I}_{k+1}\cap\widetilde{H}_k))$. 
	
	Suppose that $K$ occurs. 
	Let $w$ be an element of $\{y_z\mathrel{:} z\in D_1\}$ which satisfies $d(0,w)=\min_{z\in D_1}d(0,y_z)$. 
	It follows from the above statements (1) and (2) that every path of $\ust$ connecting the origin and $B(0,r)^c$
		includes $\gamma(0,w)$ (recall that $\partial_i B(0,r)\subset D_{k_0}$). 
	Thus, by the series law of effective resistance (see \cite{LP16} Section 2.3, for example), we have that 
	\begin{align*}
		\Reff{0}{\ust\setminus U_r}&=\Reff{0}{w}+\Reff{w}{\ust\setminus U_r}	\\
				&\ge d(0,w) \ge r^\beta/\lambda^{1+4\beta}.
	\end{align*}
	Combining this with (\ref{proof-7}) yields the desired result (\ref{Ebound}). 
\end{proof}
	
Now we are ready to prove Theorem \ref{thm-first}. 
Recall that $Z_B$ is defined in (\ref{defZB}). 
In the rest of the article, we set $B_r=B_\ust(0,r)$. 
\begin{proof}[Proof of Theorem \ref{thm-first}]
	Let us define the event $\widetilde{K}(r,\lambda)$ by
	\begin{equation}\label{proof2-1}
		\widetilde{K}(r,\lambda)=\left\{\Reff{0}{B_\ust(0,r)^c}\ge r/\lambda\right\}.
	\end{equation}
	By \cite[Proposition 4.1]{ACHS20}, there exist some $C'>0$ and $c'\in (0,1)$ such that 
	\begin{equation*}
		\prob\left(U_r\not\subset B_\ust(0,\lambda r^\beta)\right)\le C'\lambda^{-c'},
	\end{equation*}
	for all $r>1$ and $\lambda\ge 1$. 
	On the event $\{U_r\subset B_\ust(0,\lambda r^\beta)\}$, by monotonicity 
	\[
		\Reff{0}{\ust\setminus U_r}\le \Reff{0}{B_\ust(0,\lambda r^\beta)^c},
	\]
	holds (see \cite{LP16} Section 2.2, for example). 
	Thus, we have 
	\begin{align*}
		\prob\Bigl(&\Reff{0}{B_\ust(0,\lambda r^\beta)^c}<r^\beta/\lambda^{1+4\beta}\Bigr)	\notag	\\
		\le &\prob\Bigl(\Reff{0}{B_\ust(0,\lambda r^\beta)^c}<r^\beta/\lambda^{1+4\beta},\ 	\notag U_r\subset B_\ust(0,\lambda r^\beta)\Bigr)+\prob\left(U_r\not\subset B_\ust(0,\lambda r^\beta)\right)	\notag	\\
		\le &\prob\left(\Reff{0}{\ust\setminus U_r}< r^\beta/\lambda^{1+4\beta}\right)+C'\lambda^{-c'}.
	\end{align*}
	By Thorem \ref{effEball}, we obtain that 
	\begin{align*}
		\prob\left(\Reff{0}{B_\ust(0,\lambda r^\beta)^c}\ge r^\beta/\lambda^{1+4\beta}\right)
		&\ge \prob\left(\Reff{0}{\ust\setminus U_r}\ge r^\beta/\lambda^{1+4\beta}\right)-C'\lambda^{-c'}	\\
		&\ge 1-C\lambda^{-1}-C'\lambda^{-c'}.
	\end{align*}
	By reparameterizing $R=\lambda r^\beta$, and taking $C'>0$ properly, we have that 
	\begin{equation}\label{proof2-1.1}
		\prob\left(\widetilde{K}(R,\lambda)\right)\ge 1-C'\lambda^{-\frac{c'}{2+4\beta}}.
	\end{equation}
	
	Next we make use of the estimates of $E(Z_B)$ and $E(Z_B^2)$ in Section \ref{sec-defi}. 
	Since $1\le \mu_\ust(x)\le 6$ for all $x\in\zzz$, it follows from (\ref{Reff-G}) and (\ref{reff-g-2}) that 
		on the event $\widetilde{K}(r,\lambda)$, 
	\begin{equation}\label{proof2-2}
		\frac{r}{2\lambda}\le E_{0,0}(Z_{B_r}) \le 6r,
	\end{equation}
	where we plugged $\Reff{0}{B_\ust(0,r)^c}\le r$ to obtain the second inequality. 
	By reparameterization, (\ref{1st-moment}) follows. 
	
	On the other hand, since 
	\[
		\Reff{x}{B_\ust(0,r)^c}\le \Reff{x}{0}+\Reff{0}{B_\ust(0,r)^c}\le 2r,
	\]
	for $x\in B_\ust(0,r)$, plugging this into (\ref{reff-g-3}) yields that 
	\begin{equation}\label{proof2-3}
		E_{0,0}(Z_{B_r}^2)\le 144r^2+6r,
	\end{equation}
	for any realization $\ust$, which gives (\ref{2nd-moment}). 
	
	Now we will apply the second moment method to $Z_{B_r}$ on the event 
		$\widetilde{K}(r,\lambda)$. 
	By (\ref{proof2-2}) and (\ref{proof2-3}), on the event $\widetilde{K}(r,\lambda)$ we have 
	\begin{align*}
		P_{0,0}\left(Z_{B_r}\ge \frac{r}{12\lambda}\right) &\ge P_{0,0}\left(Z_{B_r}\ge \frac{1}{6}E_{0,0}(Z_{B_r})\right)	\notag	\\
			&\ge \frac{25E_{0,0}(Z_{B_r})^2}{36E_{0,0}(Z_{B_r}^2)}\ge \frac{1}{6\cdot(12\lambda)^2}. 
	\end{align*}
	By reparameterizing $\vep^{-1}=12\lambda$, we have that on $\widetilde{K}(r,\vep^{-1}/12)$,
	\begin{equation}\label{proof2-4}
		P_{0,0}(Z_{B_r}\ge \vep r)\ge \vep^2/6.
	\end{equation}
	Finally, by Markov's inequality, 
	\begin{equation*}
		P_{0,0}(Z_{B_r}\ge 72\vep^{-2}r)\ge P_{0,0}(Z_{B_r}\ge 12\vep^{-2}E_{0,0}(Z_{B_r}))\ge \vep^2/12,
	\end{equation*}
	holds on the event $\widetilde{K}(r,\vep^{-1}/12)$. 
	Combining this with (\ref{proof2-4}) gives (\ref{thm-e-1}). 
\end{proof}

We obtain the infinite collision property of the three-dimensional UST as a corollary. 
\begin{proof}[Proof of Corollary \ref{mainthm}]
	Suppose $\omega\in K(r,\vep)$ and let $\ust(\omega)$ be the corresponding realization of UST. 
	We take two simple random walks $X$ and $Y$ on $\ust(\omega)$. 
	By Theorem \ref{thm-first}, for any $N\ge 1$ and any fixed $\vep>0$, 
	\[
		P_{0,0}(Z\ge N)\ge P_{0,0}\left(Z_{B_{\vep^{-1}N}}\ge N\right)\ge \vep^2/12,
	\]
	holds. 
	By taking the limit $N\to\infty$ we obtain that $P_{0,0}(Z=\infty)\ge \vep^2/12$, 
		from which the infinite collision property of $\ust(\omega)$ follows. 
	Thus, 
	\[
		\prob(\{\ust(\omega) \mbox{ has the infinite collision property}\})\ge 1-C\vep^c.
	\]
	Since $\vep$ is arbitrary, we have that 
		$\prob(\{\ust(\omega) \mbox{ has the infinite collision property}\})=1$, 
		which completes the proof. 
\end{proof}

\begin{remark}
	We can also derive the infinite collision property of the three-dimensional UST from (\ref{proof2-1.1}) by 
		applying Corollary 3.3 of \cite{collision}. 
	In this article, we gave another proof by using quantitative estimates of the number of collisions in the 
		intrinsic ball $Z_{B_r}$.  
\end{remark}